\documentclass[a4paper, 12 pt, reqno, oneside]{amsart}
\usepackage{amsfonts, amssymb, amsmath, eucal, amsthm, stmaryrd, verbatim}
\usepackage{graphicx}
\usepackage{algorithm}
\usepackage[noend]{algpseudocode}
\usepackage{url}
\usepackage{textcomp}
\usepackage{hyperref}
\usepackage{tablefootnote}

\usepackage{color}

\usepackage[vmargin=3.2cm, hmargin=3.4cm]{geometry}

\newtheorem{lemma}{Lemma}[section]
\newtheorem{theorem}[lemma]{Theorem}
\newtheorem*{theorem*}{Theorem}
\newtheorem{proposition}[lemma]{Proposition}

\theoremstyle{definition}
\newtheorem{definition}[lemma]{Definition}

\newtheorem{remark}[lemma]{Remark}
\newtheorem{question}[lemma]{Question}

\newcommand{\R}{\mathbb{R}}
\newcommand{\Z}{\mathbb{Z}}
\newcommand{\N}{\mathbb{N}}

\newcommand{\M}{\operatorname{M}}
\newcommand{\GL}{\operatorname{GL}}

\newcommand{\dFl}{\operatorname{dFl}}
\newcommand{\tFl}{\operatorname{tFl}}

\newcommand{\torsion}{\operatorname{Torsion}}

\title{Computing Homotopy Types of Directed Flag Complexes}

\author{Dejan Govc}

\address{Institute of Mathematics, University of Aberdeen, Aberdeen,
United Kingdom AB24 3UE}

\email{dejan.govc@abdn.ac.uk}

\thanks{The author was supported by EPSRC grant EP/P025072/1}

\begin{document}

\begin{abstract}
Combinatorially and stochastically defined simplicial complexes often have the homotopy type of a wedge of spheres. A prominent conjecture of Kahle quantifies this precisely for the case of random flag complexes. We explore whether such properties might extend to graphs arising from nature. We consider the brain network (as reconstructed by Varshney \& al.) of the {\em Caenorhabditis elegans} nematode, an important model organism in biology. Using an iterative computational procedure based on elementary methods of algebraic topology, namely homology, simplicial collapses and coning operations, we show that its directed flag complex is homotopy equivalent to a wedge of spheres, completely determining, for the first time, the homotopy type of a flag complex corresponding to a brain network.

We also consider the corresponding flag tournaplex and show that torsion can be found in the homology of its local directionality filtration. As a toy example, directed flag complexes of tournaments from McKay's collection are classified up to homotopy. Moore spaces other than spheres occur in this classification. As a tool, we prove that the fundamental group of the directed flag complex of any tournament is free by considering its cell structure.
\end{abstract}

\maketitle

\thispagestyle{empty}

It has been observed that many simplicial complexes naturally occurring in combinatorial and stochastic topology have the homotopy type of a wedge of spheres (see e.g. \cite{forman, kahle, kukiela}). The reasons behind this are still far from clear. A partial explanation perhaps lies in the fact that many combinatorial methods such as shelling \cite{bjorner} and discrete Morse theory \cite{forman} are well suited to recognise this kind of homotopy type. On the other hand, structural reasons also seem to play a role. In \cite{kahle}, flag complexes (clique complexes) of Erd\H{o}s-R\'{e}nyi random graphs \cite{erdos-renyi} are studied and evidence is given towards the conjecture that, in a certain natural regime, these are almost always wedges of spheres; it is then pointed out that many simplicial complexes in combinatorics can be understood as order complexes of posets and are therefore examples of flag complexes.

A version of the flag complex construction for directed graphs called the {\em directed flag complex} has recently emerged in applications of topological methods to neuroscience \cite{frontiers}. Such complexes are examples of {\em ordered simplicial complexes} \cite{frontiers}. These are a natural generalisation of the usual concept of simplicial complex and are characterised by the property that the simplices are determined by their ordered lists of vertices, rather than their sets of vertices. {\em Flag tournaplexes} are a related, even more recent construction \cite{tournaplexes}.

When studying data in the form of directed graphs, such complexes allow bringing in a whole arsenal of topological invariants, such as homology and persistent homology \cite{oudot, edelsbrunner-harer, topforcomp}, with the aim of obtaining insight into the global structure of the data. This has been used with some success to study structure and function in the case of the Blue Brain Project reconstructions of neocortical columns of a rat \cite{frontiers}. Such computations were taken even further when \textsc{Flagser} \cite{flagser} was developed, a software package based on \textsc{Ripser} \cite{ripser} but specialised for working with directed flag complexes. A variation called \textsc{Tournser} which works for tournaplexes \cite{tournaplexes} was also presented in \cite{flagser}, as well as \textsc{Deltser} for $\Delta$-complexes \cite[Section 2.1]{hatcher}.

A much smaller example than the neocortical column of a rat is the network of chemical synapses of the {\em Caenorhabditis Elegans} nematode (commonly referred to as {\em C. Elegans}), an important model organism in biology. The interest of this network lies in the fact that it is still small enough that it can be mapped out almost completely on the biological level, as has been done in \cite{celegans}. The resulting directed graph will be referred to as the {\em C. Elegans graph} for the purposes of this paper. This graph has also been investigated topologically in \cite{frontiers}, where its directed flag complex and homology with $\Z_2$ coefficients were computed. The flag tournaplex has also been computed and the corresponding directionality distribution has been compared with the one for the rat neocortical column \cite{tournaplexes}. Remarkably, the two directionality distributions appear very similar, whereas they can be strikingly different for other graphs. %Naturally occurring networks tend to behave quite differently from Erd\H{o}s-R\'{e}nyi random graphs (see \cite{barabasi}), and this seems to transfer to the topological level as well. For instance, the directed flag complex of the C. Elegans graph is compared in \cite{frontiers} to the directed flag complex of a specific instance of an Erd\H{o}s-R\'{e}nyi directed graph with a similar number of vertices and edges and is found to have much richer topology. The former is $7$-dimensional and has nontrivial homology in every degree (see Table \ref{fig:cebetti}). The latter, by contrast, is found to be $3$-dimensional with homology only up to degree $2$. The homology in the C. Elegans case is therefore quite far from being concentrated around a single degree, so the homotopy type is far from being determined by the homology alone. A topologist might wonder how rich exactly the topology occurring in such natural systems could be. It is the aim of this paper to initiate an exploration of this question by focusing on the specific example of C. Elegans and using any means necessary to determine the homotopy type completely.

Naturally occurring networks tend to behave quite differently from Erd\H{o}s-R\'{e}nyi random graphs (see \cite{barabasi}), and this seems to transfer to the topological level as well. It is known that the Betti numbers of flag complexes of Erd\H{o}s-R\'{e}nyi random graphs tend to concentrate around a single homological degree \cite{kahle, kahleannals}. Recently, in an unpublished note \cite{lasalle}, the results of \cite{kahle} have been extended to directed flag complexes of directed Erd\H{o}s-R\'{e}nyi random graphs as well. In sharp contrast, the directed flag complex of the C. Elegans graph was found in \cite{frontiers} to have much richer topology. It is $7$-dimensional and has nontrivial homology in every degree (see Table \ref{fig:cebetti}). This is quite far from being concentrated around a single degree, so the homotopy type is far from being determined by the homology alone. A topologist might wonder how rich exactly the topology occurring in such natural systems could be. It is the aim of this paper to initiate an exploration of this question by focusing on the specific example of C. Elegans and using any means necessary to determine the \mbox{homotopy type completely.} %Avoid breaking the line.

The computations over $\Z_2$ in \cite{frontiers} do not give any information regarding torsion, so the first step is to compute integral homology. This computation reveals that there is no torsion and in particular the rational Betti numbers agree with the $\Z_2$ Betti numbers. To proceed beyond this observation, the overall strategy is to attempt simplifying the complex by splitting off as many spherical wedge summands as possible, up to homotopy. This is a well-established technique in algebraic topology, but as the complex is rather big, it needs to be done a systematic computational manner. To achieve this, an iterative procedure combining simplicial collapses to keep the complex as reduced as possible, homology computations to find candidate spherical cycles and coning operations to split them off is developed and implemented in \textsc{Mathematica}. Applied appropriately, the procedure is sufficient to simplify the complex to a single vertex, thus leading to the main result of this paper (see Theorem \ref{main}):

\begin{theorem*}
The directed flag complex of the C. Elegans graph $G$ is homotopy equivalent to a wedge of spheres.
\end{theorem*}

This appears to be the first time in the literature the homotopy type of the directed flag complex of any brain network has been completely determined. The proof consists of applying the above procedure and recording the resulting sequence of collapses and coning operations (which constitute an explicit recipe to construct the required homotopy equivalence); this relies entirely on basic homotopy theoretic principles, and hence requires the computer only for speed. The corresponding flag tournaplex is also briefly examined and torsion is found in the integral homology of certain stages of its local directionality filtration \cite{tournaplexes}, however the exact homotopy type is not analysed in this case (see Theorem \ref{torsion}):

\begin{theorem*}
Let $X^d$ be the $d$-th filtration stage of the flag tournaplex of the C. Elegans graph $G$, with respect to the local directionality filtration. For $d\in[2,10)$, the integral homology of $X^d$ contains torsion in degree $1$:
\[
\torsion(H_1(X^d))\cong\Z_3.
\]
For $d\in[20,28)$, the integral homology of $X^d$ contains torsion in degree $2$:
\[
\torsion(H_2(X^d))\cong\Z_3.
\]
In particular, for these values, $X^d$ is not a wedge of spheres.
\end{theorem*}

This appears to be the first time torsion has been found in a flag complex arising from a biological neuronal network. Finally, some of the procedures presented here were originally developed to answer certain questions arising from the study of tournaplexes \cite{tournaplexes}, one of these being how to distinguish nonisomorphic regular tournaments from one another. One possible approach is to examine the simplex counts, the Betti numbers and the homotopy types of their directed flag complexes. Such considerations eventually led to a complete classification of the tournaments, regular tournaments and doubly regular tournaments from the collection \cite{mckay} up to homotopy (see Section \ref{sec:tournaments}). Most of these are found to be wedges of spheres, but in a few cases some other Moore spaces appear. As a tool used to classify homotopy types of tournaments, properties of the cell structure of their directed flag complexes are used to prove the following result, which could be of independent interest (see Theorem \ref{pi1free}):

\begin{theorem*}
Suppose $T$ is a tournament and let $X=\dFl(T)$. Then $\pi_1(X)$ is a free group.
\end{theorem*}

The structure of the paper is as follows. Section \ref{sec:background} introduces the background definitions. Section \ref{sec:methods} explains the methods used to obtain the desired homotopy equivalences. Section \ref{sec:tournaments} is devoted to a classification of the homotopy types for the tournaments found in the collection \cite{mckay}, as well as a brief discussion of tournaplexes at the end. At the beginning, a proof of the Theorem \ref{pi1free} is also given, which serves as a useful tool in the classification. Section \ref{sec:celegans} is devoted to the main object of study, i.e. the complexes associated to the C. Elegans graph. In Appendix \ref{sec:algorithms}, an explicit description of the algorithms used is given in pseudocode, as well as some comments regarding the implementation and details of the orderings used to make the computation work.

These algorithms are not claimed to be optimal in any way. The main focus is rather to find any means possible to obtain the main result (Theorem \ref{main}), with emphasis on giving the complete description necessary for the reader to be able verify the result independently. In addition to implementing the algorithm that verifies Theorem \ref{main} in \textsc{Mathematica} (the final computation is available online, see \cite{github}, \texttt{C. Elegans.nb}), a minor modification has also been implemented which records the precise sequence of collapses and coning operations used in the process into a file. The resulting sequence of operations is available online (see \cite{github}, \texttt{sequence.dat}), so the result could in principle be verified independently by checking that each operation recorded in this output is a valid operation, without the need to implement any of the procedures described.

\section{Background}
\label{sec:background}

We start with the basic definitions. Ordered simplicial complexes can be described purely combinatorially (abstractly) in a way completely analogous to the usual definition of abstract simplicial complexes:

\begin{definition}[\cite{frontiers}]
An {\em abstract ordered simplicial complex} $X$ is a collection of nonempty finite ordered sets\footnote{An {\em ordered set} is an ordered tuple whose entries are all distinct.} that is closed under taking nonempty ordered subsets. An ordered set $\sigma$ in $X$ with $n+1$ elements is called an {\em $n$-simplex} of $X$. An element of $\sigma$ is called a {\em vertex} of $\sigma$.
\end{definition}

However, this definition also has a topological counterpart, as any such complex has a {\em geometric realisation}, which is a topological space (in fact, a CW complex) associated to it in a natural way. To define it, one can notice any such complex is naturally a semisimplicial complex (a.k.a. a $\Delta$-set) by defining the $i$-th boundary of the $n$-simplex $(v_0,\ldots,v_n)$ to be tuple obtained by removing the vertex $v_i$. A nice explanation of how to obtain a geometric realisation of a semisimplicial complex, as well as various related notions, is given in \cite{friedman}, see Example 4.5.

For the purpose at hand, the most important example of such ordered simplicial complexes is the following:

\begin{definition}[\cite{frontiers}]
Let $G$ be a directed graph\footnote{Unless stated otherwise, all graphs are assumed to be directed. Reciprocal pairs of edges are allowed, but multiple edges and loops (edges from a vertex to itself) are not.}. A {\em directed $(n+1)$-clique} in $G$ is an ordered set of vertices $(v_0,\ldots,v_n)$ such that there is an edge $v_i\to v_j$ in $G$ whenever $0\leq i<j\leq n$. The {\em directed flag complex} of $G$, denoted $\dFl(G)$ is the collection of all directed cliques in $G$.
\end{definition}

To see how this compares with ordinary flag complexes, note that the directed flag complex of a complete directed graph (i.e. graph that has a pair of reciprocal connections between any pair of vertices) on $n$ vertices contains $n!$ different $(n-1)$-simplices ($n$-cliques) which have the same underlying set of vertices. The simplices themselves correspond to the permutations of this vertex set. The complex in this particular example is known as {\em the complex of injective words on $n$ letters} and has been investigated in the literature before \cite{farmer,bjorner-wachs}. Recently, their generalisations, which provide further examples of ordered simplicial complexes, have been studied as well \cite{injective}.

However, ordered simplicial complexes can themselves be generalised. To explain how, we first recall the classical notion of tournaments. The {\em out-degree} of a vertex $v$ in a directed graph is the number of vertices $u$ such that there is an outgoing edge $v\to u$ (and analogously for the {\em in-degree}). The {\em out-degree} of a pair of vertices $\{v,w\}$ is the number of vertices $u$ such that there are outgoing vertices from both $v$ and $w$, i.e. $v\to u$ and $w\to u$ (see \cite{doubly}).

\begin{definition}
A {\em tournament} is a directed graph which has precisely one directed edge between any pair of vertices. A tournament is called {\em transitive} if it contains no directed cycles. It is called {\em regular} if the in-degree of every vertex is equal to the out-degree (equivalently, if the out-degree of every vertex is the same). It is called {\em doubly regular} (see \cite{doubly}) if it is regular and the out-degree of every pair of vertices is the same.
\end{definition}

The name comes from the fact that such graphs can be used to model a tournament where between every pair of players one player dominates the other, indicated by a directed edge between them. A classical reference on tournaments is \cite{moon}. An example of a regular tournament is the well-known game of paper-rock-scissors. Regular tournaments always have an odd number of vertices and examples for any odd number can easily be constructed. Doubly regular tournaments are somewhat trickier. They always contain $4k+3$ vertices where $k\in\N_0$. The question whether such a tournament exists for each $k$ appears to be still open. In particular, a positive answer would imply the Hadamard conjecture \cite{doubly, hadamard}.

Tournaments are directed graphs, so one can study them by looking at their directed flag complexes. In fact, this seems historically to be the one of the first special cases of directed flag complexes that has been studied in the literature \cite{burzio-demaria, demaria-kiihl, deniz}. Note that directed cliques in a tournament are precisely transitive subtournaments. Since there are no reciprocal edges, a transitive subtournament is uniquely determined by its vertex set. Therefore the directed flag complex of a tournament is a genuine (unordered) simplicial complex. As such, the directed flag complex construction in this special case does not yet exhibit the full complexity of the general construction described in \cite{frontiers}.

However, studying a tournament by associating a topological space to it is not the only thing that can be done with tournaments. Tournaments can also be viewed as building blocks of a more general kind of complex. The idea is that if directed flag complexes are built out of transitive subtournaments of the initial graph, why not build a complex that consists of all possible tournaments? This line of thinking leads to the following definition:

\begin{definition}
Let $G$ be a directed graph. The {\em flag tournaplex} of $G$, denoted $\tFl(G)$ is the collection of all tournaments contained in $G$ as subgraphs.
\end{definition}

Like directed flag complexes are examples of ordered simplicial complexes, flag tournaplexes are examples of tournaplexes \cite{tournaplexes}. In addition to the abstract definition as collections of tournaments closed with respect to taking faces (subtournaments), these too have a geometric realisation, as they are an example of semisimplicial complexes as well. An interesting feature is that various filtrations can be defined on them, arising purely from their structure. This allows them to be studied using methods such as persistent homology \cite{survey}, etc. For more details, the reader is directed to \cite{tournaplexes}.

Importantly, as all of these complexes are semisimplicial complexes, they can be studied using simplicial homology. This is useful, as it allows us to attempt understanding their homotopy type (i.e. the homotopy type of the corresponding geometric realisation) using purely computational tools. Homology is assumed to have integer coefficients throughout the paper, which is occasionally emphasised by referring to it as ``integral homology''. Reduced homology is used where convenient. The torsion subgroup (consisting of all elements of finite order) of a given abelian group $A$ is denoted by $\torsion(A)$. For an introduction to homology and other basic notions of algebraic topology, the reader is referred to one of the standard textbooks such as \cite{munkres}, \cite{hatcher} or \cite{bredon}. As a few of these basic notions play quite a prominent role in the paper, we also briefly recall them here.

For the purposes of algebraic topology, topological spaces are often too general, so one often imposes some technical conditions on the spaces under consideration. Whenever convenient, it will be assumed that the spaces under consideration are CW complexes, as on the one hand these are genuine topological spaces, and on the other hand they can be understood to subsume (ordered) simplicial complexes and other semisimplicial complexes as a special case via geometric realisation. In most cases, the reader who feels uncomfortable with the general notion of CW complex can substitute ``(ordered) simplicial complex'' wherever the concept occurs.

\begin{definition}
Let $X$ be a CW complex. The {\em cone} $CX$ over $X$ is the quotient space obtained from $X\times[0,1]$ by identifying the points of $X\times\{1\}$ into a single point; for a brief discussion of how to interpret this in the simplicial case, see the following paragraph. The {\em suspension} $SX$ of $X$ is the quotient space obtained from $X\times[-1,1]$ by identifying the points of $X\times\{1\}$ into a single point and identifying the points of $X\times\{-1\}$ into a different point (see Figure \ref{fig:cones}). The space $X$ is understood as a subspace of $CX$ and $SX$ by identifying it with $X\times\{0\}$. Given a subcomplex $A$ of $X$, one can form the space $X\cup CA$ in this way\footnote{This is a particular case of the {\em mapping cone} construction, see \cite[Example 0.13]{hatcher}.}. This operation is referred to as {\em coning off} the subspace $A$ of the space $X$.
\end{definition}

\begin{figure}[htb]
\centering
\includegraphics[height=75pt]{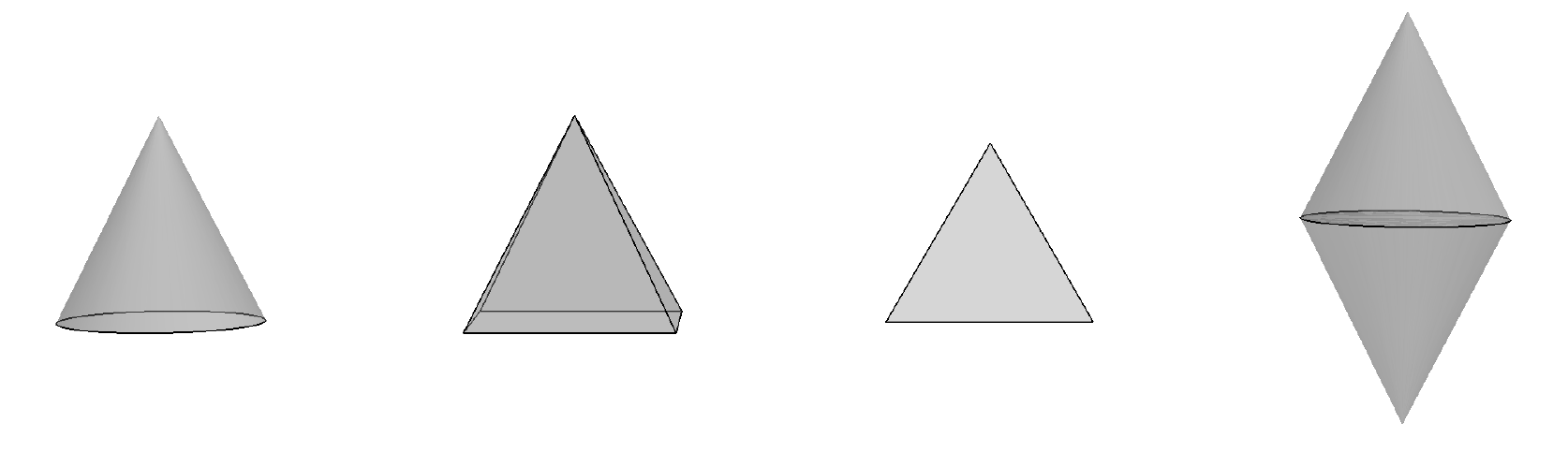}
\caption{Cone over a circle, cone over the boundary of a square, cone over an interval, suspension of a circle.}
\label{fig:cones}
\end{figure}

A suspension can also be viewed as obtained by gluing together two copies of the cone over $X$. The cone $CX$ is always contractible and $X\cap CA$ has the homotopy type of $X/A$. If $X$ is an (ordered) simplicial complex, $CX$ can also be realised as an (ordered) simplicial complex. To realise it, choose a vertex $v$ not contained in $X$, then $CX$ consists\footnote{This is a special case of the simplicial join construction which we express symbolically by writing $CX=X\ast\{v\}$ (see e.g. \cite[Chapter 8, \S 62]{munkres} or \cite[Section 2.3.4]{jonsson}).} of all possible (ordered) simplices of the form $(v_0,\ldots,v_n,v)$, where $(v_0,\ldots,v_n)$ is a simplex of $X$. Similarly, to realise $X\cup CA$, take $X$ and then add to it all possible (ordered) simplices of the form $(v_0,\ldots,v_n,v)$, where $(v_0,\ldots,v_n)$ is a simplex of $A$. Another important construction is the following.

\begin{definition}
Let $X$ and $Y$ be spaces with chosen basepoints. Then the {\em wedge sum} of $X$ and $Y$ is the space $X\vee Y$ obtained by taking the disjoint union of $X$ and $Y$ and identifying the two basepoints.\footnote{Alternatively, one could define it as the subspace of the product $X\times Y$ given by $\{x\}\times Y\cup X\times\{y\}$, where $x$ and $y$ are the chosen basepoints.} Given a finite collection $(X_i)_{i=1}^n$ of spaces with chosen basepoints, the {\em wedge sum} $\bigvee_{i=1}^n X_i$ is the space obtained by taking the disjoint union of the $X_i$ and identifying the $n$ basepoints.
\end{definition}

In most cases under consideration, the spaces will be path connected CW complexes, in which case the choice of basepoint does not matter. For this reason, basepoints will not be mentioned in the rest of the paper, but the reader should be aware that some care needs to be taken when working with disconnected complexes. Note that $\bigvee_{i=1}^n X_i$ could also be obtained (up to homeomorphism) by iterating the wedge sum operation as defined for pairs of spaces.

Wedge sum is an important operation in algebraic topology, and as already mentioned many naturally occurring spaces are homotopy equivalent to wedges of spheres. Furthermore, homology is additive with respect to the wedge sum operation in the sense that $\widetilde{H}_*(X\vee Y)=\widetilde{H}_*(X)\oplus\widetilde{H}_*(Y)$. Using this property, given a finite sequence of natural numbers $\beta_1,\ldots,\beta_n$, one can construct a space $X$ whose $i$-th Betti number for $1\leq i\leq n$ is exactly $\beta_i$. Namely, take a wedge of spheres of dimensions from $1$ to $n$ with the $i$-dimensional sphere occurring $\beta_i$ times. The converse is not true; for instance, the torus $S^1\times S^1$ and $S^1\vee S^1\vee S^2$ have the same Betti numbers but are not homotopy equivalent.

In a sense, therefore, a wedge of spheres is the simplest possible homotopy type consistent with the given sequence of Betti numbers. Notice however, that by the additivity property, the homology groups of a wedge of spheres must necessarily be free abelian, as
\[
\widetilde{H}_i(S^n)=\begin{cases}\Z;&i=n,\\
0;&\text{otherwise.}
\end{cases}
\]
This crucial property of $S^n$ is expressed by saying that $S^n$ is a Moore space $M(\Z,n)$ (see \cite[Example 2.40]{hatcher} or \cite[Section 1.3]{baues}). Not all homology groups one encounters while studying finite (ordered) simplicial complexes are free abelian, however, as a finite cell structure only leads to finitely generated abelian groups, which can also contain torsion, such as for instance $H_1(\R P^2)=\Z_2$. So, if one wishes to use the wedge sum operation as above in order to describe the ``simplest possible'' homotopy type consistent with a given sequence of finitely generated abelian homology groups, more general Moore spaces might be needed:

\begin{definition}
Suppose $A$ is an abelian group and $n\geq 1$ is an integer. Then a {\em Moore space} $M(A,n)$ is a CW complex, assumed to be simply connected if $n>1$, whose reduced homology is given by
\[
\widetilde{H}_i(X)=\begin{cases}A;&i=n,\\
0;&\text{otherwise.}
\end{cases}
\]
\end{definition}

For $n>1$, a Moore space is uniquely determined up to homotopy by the choice of $A$ and $n$ (see \cite[Example 4.34]{hatcher}). A Moore space $M(\Z_m,1)$ can be constructed by attaching a $2$-cell to $S^1$ by a map of degree $m$. For example, $M(\Z_2,1)=\R P^2$. A Moore space $M(\Z_m,n)$ can then be obtained as the $(n-1)$-fold suspension of $M(\Z_m,1)$.

One final important technique that works for ordered simplicial complexes is that of simplicial collapses: we say that the simplex $\tau$ is a {\em free face} of the complex, if it is contained in exactly one maximal\footnote{A simplex in a given complex is called {\em maximal} if it is not a proper face of any other simplex. Such simplices will also be referred to as {\em maximal faces} of the complex.} simplex $\sigma$. In this case, the operation of removing all the simplices $\rho$ that are contained in $\sigma$ and contain $\tau$, preserves the homotopy type of the complex. This operation is known as an {\em elementary collapse}\footnote{The definition in this form can be found for example in \cite{welker} or \cite{matousek} and was chosen because of its computational convenience for the purpose at hand. Note, however, that many authors additionally require that $\dim\sigma=\dim\tau+1$, see for instance \cite{cohen}.} and is a basic technique in the field of simple homotopy theory (for an introduction, see \cite{cohen}). Collapsibility for random complexes has been considered in \cite{malen, aronshtam-linial, aronshtam-linial-luczak-meshulam, cohen-costa-farber-kappeler}. For a more general overview of the many models of random complexes, see e.g. \cite{farber-mead-nowik, markstrom-pinto, costa-farber, meshulam-wallach, linial-meshulam, bobrowski-kahle-skraba, kahle-paquette-roldan} or survey papers \cite{kahlesurvey, kahlehandbook, kahlegeometric}.

\section{Methods}
\label{sec:methods}

A classical way of establishing that a CW complex is homotopy equivalent to a wedge of spheres uses the observation that the quotient space $X/B$ of a CW complex $X$ by a contractible subcomplex $B$ is homotopy equivalent to $X$, i.e. $X\simeq X/B$ (see e.g. \cite[Proposition 0.17]{hatcher}, \cite[Lemma 2.2]{bjorner-walker} or \cite[Lemma 4.1.5]{matousek}; see also Figure \ref{fig:easywedge}).

\begin{lemma}
\label{easywedge}
Suppose $X$ is a CW complex and $e_1,\ldots,e_k$ in $X$ are maximal cells such that $X-\{e_1,\ldots,e_k\}$ is contractible. Then, writing $d_i=\dim e_i$,
\[
X\simeq\bigvee_{i=1}^k S^{d_i}.
\]
\end{lemma}

\begin{proof}
Since $B:=X-\{e_1,\ldots,e_k\}$ is contractible, $X\simeq X/B\simeq\bigvee_{i=1}^k S^{d_i}$.
\end{proof}

\begin{figure}[htb]
\centering
\includegraphics[height=75pt]{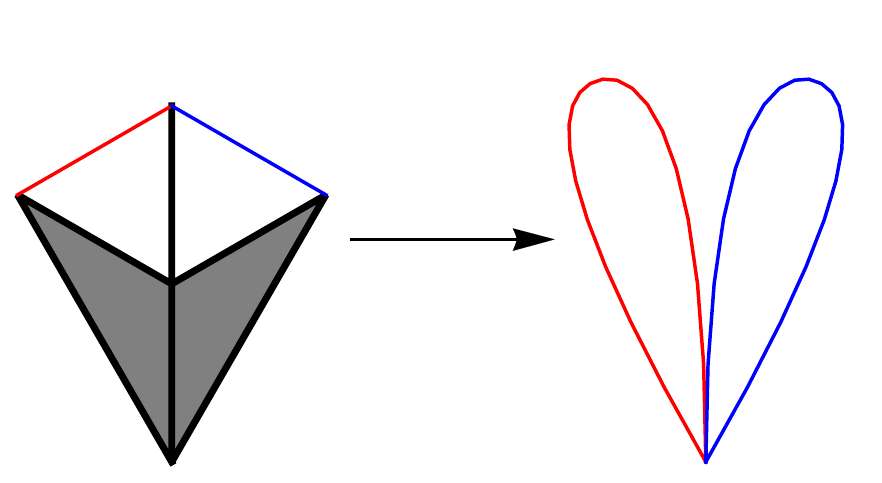}
\caption{The homotopy equivalence of Lemma \ref{easywedge}.}
\label{fig:easywedge}
\end{figure}

This suggests the following strategy to prove that a directed flag complex has the homotopy type of a wedge of spheres: collapse $X=\dFl(G)$ as much as possible, and apply Lemma \ref{easywedge}. Note that this relies on being able to find the relevant set of maximal simplices. One can use a heuristic greedy approach:
\begin{itemize}
\item Look at each maximal simplex in the complex and check if removing it would decrease the Betti number in the corresponding dimension by one (and preserve the other Betti numbers). If so, remove the simplex.
\item In the favourable case, this eventually leads to an acyclic complex. Check whether this complex is contractible (e.g. by showing that it is collapsible or that the fundamental group is trivial). If so, conclude that the initial complex is a wedge of spheres.
\end{itemize}
In the first step we are only removing the simplex, while leaving its faces intact. In principle, this could cause new simplices to become maximal, namely some of these faces. However, this does not actually happen because the condition on Betti numbers ensures that each face is also contained in at least one other maximal simplex. %This can be seen as a special case of Lemma \ref{genwedge}.
So, indeed, when the procedure terminates we will have selected a set of simplices that are maximal in the original complex.

A concrete procedure based on this outline called \textsf{pop-everything} is described in Appendix \ref{sec:algorithms}. Note that this approach works only under the most favourable circumstances, so alternative ideas are needed to treat other cases. The following lemma is a straightforward generalisation of Lemma \ref{easywedge}, but significantly more powerful (see also Figure \ref{fig:genwedge}).

\begin{lemma}
\label{genwedge}
Let $X$ be a CW complex and $A$ a subcomplex of $X$. Suppose $e_1,\ldots,e_k$ in $A$ are cells which are maximal in $X$ such that $A-\{e_1,\ldots,e_k\}$ is contractible. Then, writing $d_i=\dim e_i$,
\[
%A\simeq\bigvee_{i=1}^k S^{d_i}\quad\text{and}\quad
X\simeq (X-\{e_1,\ldots,e_k\})\vee\bigvee_{i=1}^k S^{d_i}\simeq (X\cup CA)\vee\bigvee_{i=1}^k S^{d_i}.
\]
\end{lemma}

\begin{proof}
Let $B:=X-\{e_1,\ldots,e_k\}$. Then $A\cap B=A-\{e_1,\ldots,e_k\}$ is contractible and $A\cup B=X$, therefore:
\[
X=A\cup B\simeq\frac{A\cup B}{A\cap B}\simeq\frac{A}{A\cap B}\vee\frac{B}{A\cap B}\simeq A\vee B.
\]
Now note that by Lemma \ref{easywedge}, we have
\[
A\simeq\bigvee_{i=1}^k S^{d_i}
\]
and since $A\cap B=A-\{e_1,\ldots,e_k\}$ is contractible, we also have
\[
B=X-\{e_1,\ldots,e_k\}\simeq \frac{X-\{e_1,\ldots,e_k\}}{A-\{e_1,\ldots,e_k\}}\simeq X/A\simeq X\cup CA.
\]
\end{proof}

\begin{figure}[htb]
\centering
\includegraphics[height=135pt]{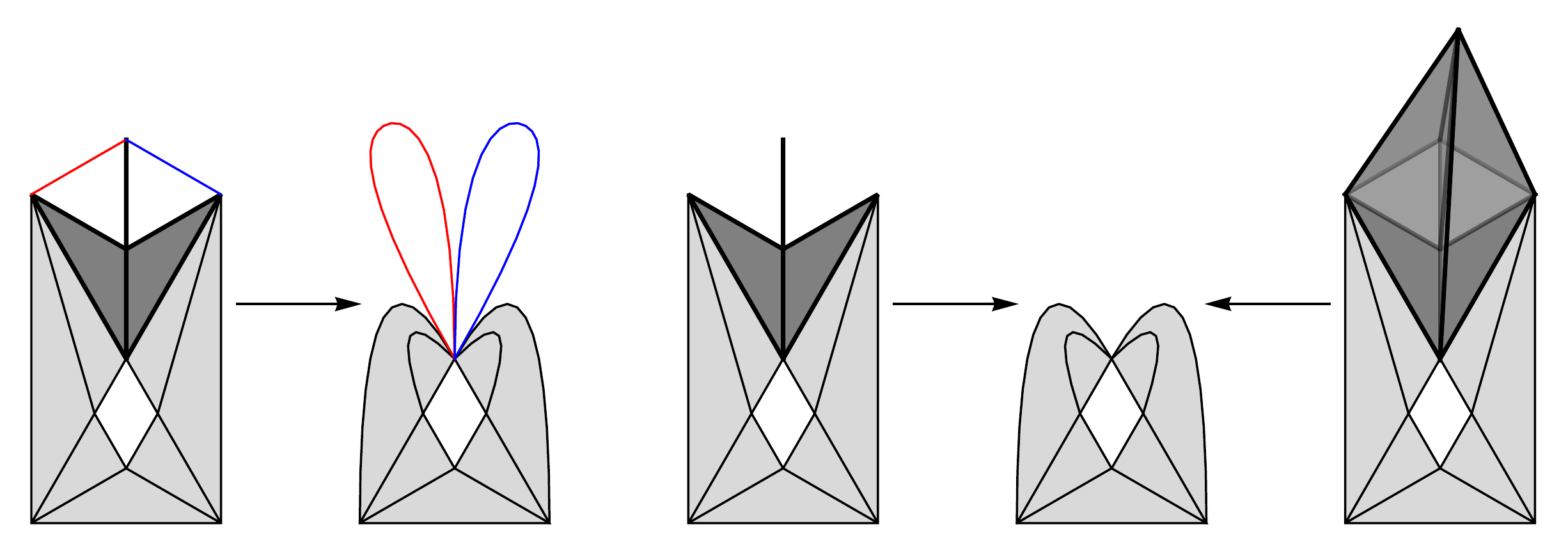}
\caption{The homotopy equivalences of Lemma \ref{genwedge}. The first arrow is the homotopy equivalence of $A\cup B$ and $\frac{A}{A\cap B}\vee\frac{B}{A\cap B}$, which in turn is homotopy equivalent to $A\vee B$ (not pictured). The other two arrows show how $B\simeq\frac{B}{A\cap B}\simeq X\cup CA$.}
\label{fig:genwedge}
\end{figure}

This allows the strategy described above to be generalised to an iterative approach to proving that the directed flag complex $X=\dFl(G)$ has the homotopy type of a wedge of spheres:
\begin{itemize}
\item To initialise, we define $X_0=X$.
\item At the $j$-th step we first collapse the complex $X_j$ as much as possible to a complex $X_j'$ and then use Lemma \ref{genwedge} to simplify it to a complex $X_{j+1}$, by splitting off a wedge of spheres wedge summand (up to homotopy).
\item If this eventually leads to a contractible complex, we can conclude that $X$ itself is a wedge of spheres.
\end{itemize}
Here, the second bullet point is performed by finding a subspace $A$ of $X_j'$ that has the homotopy type of a wedge of spheres and then show that removing some simplices $e_1,\ldots,e_k$ from $A$ which are maximal in $X_j'$ yields a contractible subcomplex (this essentially amounts to using the procedure arising from Lemma \ref{easywedge} on $A$). Then we can proceed in one of two ways provided by Lemma \ref{genwedge}:
\begin{itemize}
\item Either the simplices $e_1,\ldots,e_k$ are removed from $X_j'$, yielding the simpler complex $X_{j+1}=X_j'-\{e_1,\ldots,e_k\}$. We refer to this as ``popping the simplices''.
\item Alternatively, $A$ is coned off, yielding the simpler complex $X_{j+1}=X_j'\cup CA$. We refer to this as ``coning off the subcomplex''.
\end{itemize}
It is in general far from clear how to find a suitable subcomplex $A$ at each step. In the cases under consideration, however, this can be achieved using various heuristics to find $A$, including the following:
\begin{itemize}
\item Find nice homology cycles in $X_j'$. e.g. whose supports have the homotopy type of a sphere and try to use these supports.
\item Try to use the union of all top-dimensional simplices in $X_j'$. In case this union is disconnected, try using its components instead.
\item Build $A$ from the simplices of $X_j'$ using a greedy approach, one simplex at a time, each time checking that the reduced homology is either trivial or concentrated in a single degree.
\end{itemize}
It is important to note that the outcome of any such procedure will depend heavily on the sequence of collapses chosen. Even if starting with a collapsible space, choosing the ``wrong'' sequence of collapses can lead to a non-collapsible space. See for instance \cite{lofano-newman}, which shows that attempts to collapse a standard simplex can already get stuck and in fact seem increasingly likely to do so as the dimension of the simplex increases. Furthermore, the outcome will also depend on the choice of the subcomplex $A$ (it could be homotopy equivalent to a single sphere or a wedge of many spheres of possibly different dimensions), as well as the choice of whether to cone this subcomplex off, or pop its simplices, in which case the exact choice of simplices to pop will also matter.

A concrete procedure based on the above outline called \textsf{cone-and-collapse} is described in Appendix \ref{sec:algorithms}. Note that this particular procedure is based on coning off the subcomplex $A$ at each step rather than popping the simplices as this turned out to be less likely to get stuck for the complexes it was tested on (the drawback is that a coning operation will temporarily increase the size of the complex thus making it more computationally expensive). The procedure \textsf{cone-and-collapse} is already powerful enough to show that the directed flag complex of the C. Elegans connectome has the homotopy type of a wedge of spheres, however, certain choices (fully detailed in Appendix \ref{sec:algorithms}) have to be made ``correctly'' to avoid the issues described in the previous paragraph.

More generally, Lemma \ref{genwedge} can also be used to simplify the topology of spaces which are not necessarily wedges of spheres, by splitting off as many spherical wedge summands as possible and leaving a remainder whose topology can then be analysed separately.

Despite its apparent versatility, applying the procedure just described can be time consuming, so it is beneficial to avoid it when shortcuts are available. In certain cases, it is possible to completely determine the homotopy type already from the homology of the space. One sufficient condition for this to occur is described in the following proposition taken from \cite[Example 4C.2]{hatcher} (note the similarity to uniqueness of Moore spaces mentioned in Section \ref{sec:background}):

\begin{proposition}\label{moorewedge}
Let $n>1$ and $k,l\geq 0$. Let $A$ be a finite abelian group. Suppose $X$ is a simply connected CW complex whose reduced homology is given by
\[
\widetilde{H}_i(X)=\begin{cases}\Z^k\oplus A;&i=n,\\
\Z^l;&i=n+1,\\
0;&\text{otherwise.}
\end{cases}
\]
Then
\[
X\simeq\bigvee_k S^n\vee\bigvee_l S^{n+1}\vee M(A,n),
\]
where $M(A,n)$ is the Moore space of $A$ in degree $n$. In particular, if $H_n(X)$ is also free (i.e. $A=0$), $X$ is a wedge of spheres.
\end{proposition}

More generally, it is known that a simply connected space is homotopy equivalent to a wedge of Moore spaces if and only if the Hurewicz homomorphism $h_n:\pi_n(X)\to H_n(X)$ is split surjective for every $n$ \cite[Proposition 2.6.15]{baues}. Whether this can be used in a computational setting is not completely clear. If it can, it could potentially lead to an approach to detecting wedges of Moore spaces which avoids simplicial collapses altogether.

\section{Tournaments}
\label{sec:tournaments}

This section summarises some results obtained by testing these techniques on directed flag complexes $\dFl(T)$ of tournaments $T$. As explained in Section \ref{sec:background} these appear to be historically one of the first special cases of directed flag complexes considered in the literature. In the following, a description in terms of wedge sums is given for the various kinds of homotopy types that occur. The exact numbers of wedge summands are not listed here for every single special case, but see Remark \ref{github}. Note that there are many cases where the directed flag complexes of two nonisomorphic tournaments have the same homotopy type, but can nonetheless be distinguished by the number of simplices occurring in them, or vice versa.

\begin{remark} \label{github}
For the interested reader, the exact numbers of wedge summands, or equivalently in the case at hand, Betti vectors and torsion coefficients, are available online \cite{github} in the file \texttt{tournaments.dat}. There is also a \textsc{Mathematica} notebook there, called \texttt{Tournaments.nb}, which reproduces a part of these computations from scratch. The ones that are not reproduced there, as well as the corresponding face vectors (simplex counts) of the various complexes, if needed, are available from the author upon request.
\end{remark}

We begin with the following observation, which curiously does not seem to have appeared in the literature, despite a few special cases being proved in \cite{burzio-demaria-2}:

\begin{theorem}
\label{pi1free}
Suppose $T$ is a tournament and let $X=\dFl(T)$. Then $\pi_1(X)$ is a free group.
\end{theorem}

\begin{proof}
Let $Y=X^{(2)}$ be the $2$-skeleton of $X$. Then $\pi_1(Y)\cong\pi_1(X)$ (see e.g. \cite[Proposition 1.26]{hatcher}). Note that $Y$ consists of vertices $v_1,\ldots,v_n$, a single edge $e_{ij}$ between $v_i$ and $v_j$ for any $i<j$ (note that this edge may have either the orientation $v_i\to v_j$ or $v_j\to v_i$ in the tournament itself, but this information is irrelevant for the topology of the complex) and a triangle $t_{ijk}$ for any $i<j<k$ such that the tournament induced on the vertices $v_i,v_j$ and $v_k$ is transitive.

We can assume that the vertices are ordered so that there exists $m\in\{1,\ldots,n\}$ such that the edge orientations in $T$ are given as $v_i\to v_1$ for $i\leq m$ and $v_1\to v_i$ for $i\geq m+1$. In other words, we partition the vertices according to whether they have an incoming or an outgoing edge to $v_1$. This induces a partition of the edges into five subsets:
\begin{itemize}
\item with $v_1$ as endpoint: $\mathcal E=\{e_{1j}\mid 1<j\leq n\}$,
\item incoming to incoming: $\mathcal X=\{e_{ij}\mid 1<i<j\leq m\}$,
\item outgoing to outgoing: $\mathcal Y=\{e_{ij}\mid m<i<j\}$,
\item incoming to outgoing: $\mathcal Z=\{e_{ij}\mid i\leq m<j\text{ and }v_i\to v_j\}$,
\item outgoing to incoming: $\mathcal W=\{e_{ij}\mid i\leq m<j\text{ and }v_j\to v_i\}$.
\end{itemize}

We define a $2$-dimensional subcomplex $A$ in $Y$ consisting of all vertices of $Y$, all edges in $\mathcal E\cup \mathcal X\cup \mathcal Y\cup \mathcal Z$ and all triangles $t_{1jk}$ of $Y$ with $1<j<k$. Note that topologically, this is a cone over $\mathcal X\cup \mathcal Y\cup \mathcal Z$ with apex $v_1$, so it is a contractible subcomplex. Therefore $\pi_1(Y)\cong\pi_1(Y/A)$. The complex $Y/A$ has a single vertex $v$, edges given as elements of $\mathcal W$ and the $2$-cells as $t_{ijk}$ with $1<i<j<k$.

From this cell structure, we can write down the group presentation as follows:
\[
\pi_1(Y/A)=\langle \mathcal W\mid \tau_{ijk}, 1<i<j<k\rangle
\]
where $\tau_{ijk}$ is the word describing the attaching map of the triangle $t_{ijk}$ (with $1<i<j<k$) in terms of the generators $\mathcal W$. More specifically, we have four possibilities for $\tau_{ijk}$ depending on the position of $m$ with respect to $i,j$ and $k$:
\begin{itemize}
\item If $i<j<k\leq m$, all the edges of $t_{ijk}$ lie in $\mathcal X$, and therefore in $A$, so the corresponding word $\tau_{ijk}$ is trivial.
\item If $i<j\leq m<k$, the edge $e_{ij}$ lies in $\mathcal X$, but the other two edges of $t_{ijk}$ lie in $\mathcal Z\cup\mathcal W$, so the corresponding word $\tau_{ijk}$ is either trivial, consists of a single element $e_{ij}\in\mathcal W$ or is a product of two elements $e_{i_1j}e_{i_2j}^{-1}$ where $e_{i_1j},e_{i_2j}\in\mathcal W$.
\item If $i\leq m<j<k$, the edge $e_{jk}$ lies in $\mathcal Y$, but the other two edges of $t_{ijk}$ lie in $\mathcal Z\cup\mathcal W$, so the corresponding word $\tau_{ijk}$ is either trivial, consists of a single element $e_{ij}\in\mathcal W$ or is a product of two elements $e_{ij_1}e_{ij_2}^{-1}$ where $e_{ij_1},e_{ij_2}\in\mathcal W$.
\item If $m<i<j<k$, all the edges of $t_{ijk}$ lie in $\mathcal Y$, and therefore in $A$, so the corresponding word $\tau_{ijk}$ is trivial.
\end{itemize}

This means that each relator either corresponds to a generator in $\mathcal W$ being trivial or to two generators in $\mathcal W$ being equal. Applying the appropriate sequence of Tietze transformations therefore shows that the group is free.
\end{proof}

In particular, this means that in order to verify that $X=\dFl(T)$ is simply connected for a tournament $T$, it suffices to check that $H_1(X)$ is trivial, as $H_1(X)$ is the abelianisation of $\pi_1(X)$ (see \cite[Theorem 2A.1]{hatcher}). Combined with Proposition \ref{moorewedge} this allows us in some cases to determine the homotopy types of directed flag complexes of a large number of tournaments just by computing their homology.

\subsection{Regular Tournaments}

In this section, a classification is given of the homotopy types of directed flag complexes of regular tournaments with up to $13$ vertices whose isomorphism types are listed in the collection \cite{mckay}.

\begin{proposition}
Suppose $T$ is a regular tournament with $3\leq n\leq 13$ vertices ($n$ odd). Then $T$ is a wedge of Moore spaces. The number of different isomorphism types of $T$ versus the number of different homotopy types of $\dFl(T)$ is represented in the following table:
\[
\begin{array}{|c|cccccc|}
\hline
n&3&5&7&9&11&13\\
\hline
\text{\# isomorphism types}&1&1&3&15&1223&1495297\\
\hline
\text{\# homotopy types}&1&1&3&8&40&183\\
\hline
\end{array}
\]
Furthermore, these directed flag complexes have the following properties:
\begin{itemize}
\item For $n=3$ and $n=5$, there is a unique homotopy type, namely $S^1$.
\item For $n=7$, the homotopy types are $S^1$, $\bigvee_3 S^2$ and $\bigvee_6 S^2$.
\item For $n=9$, there is exactly one complex with the homotopy type of $S^1$ and exactly one complex with the homotopy type of $S^1\vee\bigvee_3 S^2\vee\bigvee_3 S^3$. The remaining complexes are wedges of a positive number of $2$-spheres. This is also the first time the classification up to isomorphism differs from the one up to homotopy.
\item For $n=11$, there is exactly one complex which is not simply connected, with the homotopy type of $S^1$. There are exactly two complexes which are contractible. Among the other complexes, exactly two contain torsion; their homotopy types are $M(\Z_2,2)\vee\bigvee_5 S^2$ and $M(\Z_2,2)\vee\bigvee_{10} S^2$. The remaining complexes are wedges of various numbers of $2$-spheres and $3$-spheres.
\item For $n=13$, there is exactly one complex which is not simply connected, with the homotopy type of $S^1$. There are $10080$ complexes which are contractible. There are $141$ complexes containing torsion, each time in the form of a single $M(\Z_2,2)$ wedge summand. Among the complexes containing torsion, there are $14$ different homotopy types, and in particular, $8$ complexes are homotopy equivalent to $M(\Z_2,2)$. All the other occurring wedge summands are $2$-, $3$- and $4$-spheres.
\end{itemize}
\end{proposition}

\begin{proof}
We can compute integral homology of all of these complexes, for example using \textsc{Mathematica} (see Remark \ref{github}). For each $n$ we find exactly one complex with the homology of the circle\footnote{This particular example in fact agrees with $\mathcal N(n,\frac{n-1}{2})$ of \cite[Definition 3.2]{nervecircular}.} and that complex can be shown to collapse onto the circle. The only other complex with nontrivial $H_1$ occurs for $n=9$. We can show that this complex is homotopy equivalent to $S^1\vee\bigvee_3 S^2\vee\bigvee_3 S^3$ by applying the procedure \textsf{cone-and-collapse} (based on Lemma \ref{genwedge}) to it.

All the remaining complexes have vanishing $H_1$. Since their fundamental group is free by Theorem \ref{pi1free}, this means that all other complexes are simply connected. All homology in degrees $\geq5$ vanishes and the only torsion that occurs in any of the cases is $\torsion(H_2)=\Z_2$. The cases where $H_4$ vanishes are therefore wedges of Moore spaces by Proposition \ref{moorewedge}.

This leaves $211$ complexes remaining with nontrivial $H_4$ (each of these has $13$ vertices). This is a number small enough that treating them using the procedures \textsf{pop-everything} (based on Lemma \ref{easywedge}) and \textsf{cone-and-collapse} becomes feasible. In all $211$ cases, at least one of these procedures is sufficient to establish that the complex at hand is a wedge of spheres.

This concludes the proof that up to homotopy, all these complexes are indeed wedges of Moore spaces. The remaining properties in the statement can therefore be read off from their homology which was computed in the beginning.
\end{proof}

\subsection{Doubly Regular Tournaments}

Doubly regular tournaments are examined next. The collection \cite{mckay} lists them completely up to $27$ vertices.

\begin{proposition}
Let $3\leq n\leq 27$, $n\equiv 3 \pmod 4$. The number of isomorphism types vs. homotopy types of directed flag complexes of doubly regular tournaments is as follows:
\[
\begin{array}{|c|ccccccc|}
\hline
n&3&7&11&15&19&23&27\\
\hline
\text{\# isomorphism types}&1&1&1&2&2&37&722\\
\hline
\text{\# homotopy types}&1&1&1&1&2&11&109\\
\hline
\end{array}
\]
Furthermore, the homotopy types can be listed exactly:
\begin{itemize}
\item For $n=3$, there is one complex with homotopy type $S^1$.
\item For $n=7$, there is one complex with homotopy type $\bigvee_6 S^2$.
\item For $n=11$, there is one complex with homotopy type $M(\Z_2,2)\vee\bigvee_{10}S^2$.
\item For $n=15$, there are two complexes with homotopy type $\bigvee_{70}S^3$.
\item For $n=19$, there are two complexes with homotopy types $\bigvee_{154}S^3\vee S^4$ and $\bigvee_{135}S^3$.
\item For $n=23$, there are $28$ complexes with homotopy types $\bigvee_{k}S^3$ for some $k\in\{90,92,93,97,98,101\}$ (occurring $6$, $2$, $4$, $8$, $6$ and $2$ times, respectively). There are also two complexes with the homotopy type $M(\Z_2,3)\vee\bigvee_{88}S^3$, two complexes with the homotopy type $M(\Z_2,3)\vee\bigvee_{91}S^3$ and two complexes with the homotopy type $\bigvee_{2}M(\Z_2,3)\vee\bigvee_{89}S^3$. Finally, there are two complexes with the homotopy type
\[
\bigvee_{12}M(\Z_2,3)\vee\bigvee_{10}M(\Z_4,3)\vee\bigvee_{67}S^3\vee S^4
\]
and one complex with the homotopy type
\[
M(\Z_{23},3)\vee\bigvee_{23}S^3\vee\bigvee_{45}S^4.
\]
\item For $n=27$, there are $719$ complexes with the homotopy type $\bigvee_{k}S^4$ for various $k$ between $125$ and $410$. There are two complexes with the homotopy type $S^3\vee\bigvee_{456}S^4$ and one with the homotopy type $M(\Z_2,3)\vee S^3\vee\bigvee_{729}S^4$.
\end{itemize}
\end{proposition}

\begin{proof}
Compute the integral homology, e.g. using \textsc{Mathematica} (see Remark \ref{github}). For the case $n=27$, the chain complexes are quite big already, but the method of \cite{leon} based on algebraic Morse theory can be used to reduce the computation time. In every case ($n>3$), the conditions of Proposition \ref{moorewedge} are satisfied and the numbers of wedge summands of each type can be read off from the homology.
\end{proof}

\subsection{General Tournaments}

The collection \cite{mckay} also contains the isomorphism types of general tournaments up to $10$ vertices, which can also be classified completely up to homotopy.

\begin{proposition}
Let $2\leq n\leq 10$. The number of isomorphism types vs. homotopy types of directed flag complexes of general tournaments on $n$ vertices is as follows:
\[
\begin{array}{|c|ccccccccc|}
\hline
n&2&3&4&5&6&7&8&9&10\\
\hline
\text{\# isomorphism types}&1&2&4&12&56&456&6880&191536&9733056\\
\hline
\text{\# homotopy types}&1&2&2&3&6&11&21&53&114\\
\hline
\end{array}
\]
These complexes are all homotopy equivalent to wedges of spheres of various dimensions, where $S^1$ appears as a wedge summand at most once in each case.
\end{proposition}

\begin{proof}
Compute the integral homology, e.g. using \textsc{Mathematica} (see Remark \ref{github}). In every case the homotopy type is uniquely determined by the homology, as the vast majority of cases satisfy the conditions of Proposition \ref{moorewedge} and the rest can be shown to be wedges of spheres using the procedures \textsf{pop-everything} and \textsf{cone-and-collapse}. For $n=10$, the reduction from \cite{leon} can be used to reduce the computation time. Even so, the total computation time for this case was about two days on a laptop PC with a 2-core processor and 32GB of RAM.
\end{proof}

\subsection{Tournaplexes}

This section mostly serves to fill in some details regarding Table 3.1 of the paper \cite{tournaplexes} which introduced the notion of tournaplex, but also as an example application of the methods beyond the case of ordered simplicial complexes, as tournaplexes are a generalisation of them. As a particular example, to a directed graph $G$, one can associate the so-called {\em flag tournaplex} $\tFl(G)$, which is a semisimplicial complex whose simplices are precisely all subtournaments of $G$. In particular, it contains the directed flag complex $\dFl(G)$, whose simplices are the transitive tournaments of $G$, as a subcomplex. An interesting feature of tournaplexes is that they can be equipped with various filtrations arising from their structure, known as {\em directionality filtrations}.

In \cite{tournaplexes}, the techniques of this paper have been used to demonstrate that a certain bifiltration of tournaplexes can carry strictly more information than can be obtained by just combining the information arising from the two filtrations corresponding to it separately and not taking into account how they interact. The examples found have been analysed by determining the homotopy type of each bifiltration stage. These can be seen in Table 3.1 of \cite{tournaplexes}. Here, a brief explanation is given of how that table was obtained. First, compute the simplicial homology. Upon performing this computation, one notices that most of the complexes are simply connected, either by applying Theorem \ref{pi1free} when possible or by noticing that the $2$-skeleton agrees with the $2$-skeleton of the $n$-simplex for most other stages of the bifiltration. In either of these cases, one can then use Proposition \ref{moorewedge} to conclude. The complexes which do not fit under either of these cases are very small and can be treated separately in a variety of possible ways. For instance, it is possible to collapse all simplices of dimension $\geq 2$ in each of them, from which one can then easily determine the homotopy type. %Note that all stages of the bifiltration are genuine simplicial complexes as they are complexes of subtournaments of a tournament, i.e. faces of a simplex.
The actual computations used are available online \cite{github} in the file \texttt{Bifiltration Example.nb}.

\section{C. Elegans}
\label{sec:celegans}

Directed flag complexes can be used in neuroscience as a way to understand the global structure of graphs arising from brain data. In \cite{frontiers}, the network of chemical synapses in the brain of a C. Elegans nematode (as reconstructed by \cite{celegans}) has been analysed in this way, and the homology of the directed flag complex was computed (see Table \ref{fig:cebetti}).

\begin{table}[htb]
\centering
\begin{tabular}{|c||c|c|c|c|c|c|c|c|}
\hline
\# simplices&279&2194\footnotemark&4320&4902&4449&2709&901&155\\
\hline
\# Betti&1&183&249&134&105&63&19&5\\
\hline
\end{tabular}
\vskip.1in
\caption{Simplex counts and Betti numbers of the C. Elegans directed flag complex.}
\label{fig:cebetti}
\end{table}\footnotetext{In \cite{frontiers}, this is listed as $2199$, which appears to be a typo.}

This leaves open the question of what the actual homotopy type is and this is addressed next:

\begin{theorem}
\label{main}
The directed flag complex of the C. Elegans graph $G$ is homotopy equivalent to a wedge of spheres.
\end{theorem}

\begin{proof}
This is demonstrated by a direct application of the procedure \textsf{cone-and-collapse}, implemented in \textsc{Mathematica}. The implementation is available online \cite{github}, where the procedure is implemented in the file \texttt{Main Functions.nb} and the computation based on it is done in the file \texttt{C. Elegans.nb}. A complete list of simplices in the directed flag complex\footnote{The underlying graph is also implicitly recorded in this information as the $1$-skeleton.} together with the complete sequence of collapses and coning operations is provided in the file \texttt{sequence.dat}\footnote{Note that the data in this file alone is sufficient to verify the result, as it constitutes a recipe to construct an explicit homotopy equivalence.}, which was created using the file \texttt{C. Elegans to File.nb}.

The details of the procedure \textsf{cone-and-collapse} are described in pseudocode in Appendix \ref{sec:algorithms}, with some further details regarding the specific implementation used given in Appendix \ref{sec:implementation} and some preprocessing steps that are needed to get the input into a form amenable to the procedure described in Appendix \ref{sec:preprocessing} (a specific ordering of the vertices and edges in the graph needs to be used to avoid the algorithm getting stuck).
\end{proof}

A completely analogous result holds for the undirected flag complex of the undirected graph obtained from $G$ by forgetting the orientations of edges (duplicate pairs of edges arising from reciprocal pairs are counted as single edges). This computation is also available online \cite{github}, \texttt{C. Elegans Undirected.nb}. The simplex counts and Betti numbers can be seen in Table \ref{fig:ucebetti}.

\begin{table}[htb]
\centering
\begin{tabular}{|c||c|c|c|c|c|c|c|c|}
\hline
\# simplices&279&1961&2858&1891&869&278&50&4\\
\hline
\# Betti&1&162&83&/&/&/&/&/\\
\hline
\end{tabular}
\vskip.1in
\caption{Simplex counts and Betti numbers of the undirected C. Elegans flag complex. Interestingly, the Betti numbers are much lower in this case and only go up to degree 2.}
\label{fig:ucebetti}
\end{table}

One could in principle also use these methods to analyse the flag tournaplex $\tFl(G)$ of the C. Elegans graph $G$. The complete analysis of the homotopy type has not been performed and will not be given here. However, integral homology has been computed and is again torsion-free. The Betti numbers and simplex counts can be seen in Table \ref{fig:tcebetti}. Integral homology has also been computed for the various stages of the corresponding {\em local directionality filtration} (see \cite[Definition 3.1 and Lemma 3.2]{tournaplexes}). It turns out that for specific filtration values these filtration stages do contain torsion:

\begin{theorem}
\label{torsion}
Let $X^d$ be the $d$-th filtration stage of the flag tournaplex of the C. Elegans graph $G$, with respect to the local directionality filtration. For $d\in[2,10)$, the integral homology of $X^d$ contains torsion in degree $1$:
\[
\torsion(H_1(X^d))\cong\Z_3.
\]
For $d\in[20,28)$, the integral homology of $X^d$ contains torsion in degree $2$:
\[
\torsion(H_2(X^d))\cong\Z_3.
\]
In particular, for these values, $X^d$ is not a wedge of spheres. % Open question: what are the homotopy types of the various stages of the bifiltration?
\end{theorem}

\begin{proof}
Compute integral homology, e.g. using \textsc{Mathematica}. The computation is available online \cite{github} in the file \texttt{C. Elegans Tournaplex.nb}.
\end{proof}

\begin{table}[htb]
\centering
\begin{tabular}{|c||c|c|c|c|c|c|c|c|}
\hline
\# simplices&279&2194&4836&7662&13110&20530&22504&11520\\
\hline
\# Betti&1&164&261&387&574&734&1924&2652\\
\hline
\end{tabular}
\vskip.1in
\caption{Simplex counts and Betti numbers of the C. Elegans flag tournaplex.}
\label{fig:tcebetti}
\end{table}

This torsion occurs in somewhat low stages of the filtration. For example, $X^d$ for $d\in[2,10)$ consists precisely of the graph $G$ with a triangle glued in for every directed $3$-cycle (i.e. regular $3$-tournament). The significance of this torsion is unclear. It would be interesting to find examples where the full directed flag complex or flag tournaplex has a homotopy type which is not that of a wedge of spheres, without filtering them in any way. For example, the following question still remains open:

\begin{question}
What are the homotopy types of the directed flag complexes of the BBP microcircuits analysed in \cite{frontiers}? In particular, can torsion be found in the integral homology of any of them? Integral homology has been computed for several small subgraphs (about 2000 vertices) consisting of particular types of neurons and thus far none of them contained torsion.
\end{question}

If some of these complexes turn out not to be wedges of spheres, there could be algebraic invariants showing this (a potentially promising idea in this direction is developed in \cite{anibal}). But for those that are, one possibility would be to try and extend the methods of the present paper. It should be noted, however, that the directed flag complex treated here was small enough (both in dimension as well as number of simplices) that the phenomenon of ``getting stuck while collapsing'' was still possible to bypass.

\begin{question}
Is there a more systematic way of dealing with the issue of getting stuck? For instance, can simplicial expansions be used as a way of ``backtracking''? Without a systematic way of collapsing, the issue seems to become prohibitive for larger complexes \cite{lofano-newman}, but is there at least something that works for a large subclass of relatively small complexes (e.g. at most $30000$ vertices and dimension at most $10$)?
\end{question}

\begin{remark}\label{rmk:othermethods}
Note that it might be beneficial to make more use of fundamental groups when recognising contractible complexes, although this is also known to be undecidable in general (see \cite{tancer}, Appendix A). In some early versions of the main computation, some additional methods have been used such as homotopy colimits to blow up the complex when other procedures get stuck, the Seifert-van Kampen theorem as a method of checking that the complex is simply connected, relative homology for some intermediate heuristic computations, and a greedy method of building the subcomplex $A$ to be coned off, as described in the third bullet point of the paragraph discussing various heuristics to construct $A$ in Section \ref{sec:methods} (if needed, some of these are available upon request from the author). It would be interesting to see whether such methods could be reincorporated into the procedure to make it more powerful.
\end{remark}

\begin{remark}
One final amusing property of the C. Elegans graph $G$ is that is has a nontrivial automorphism group. This can be seen as follows: first compute the bidegrees of the vertices. Then compute the sets of bidegrees of all in-neighbours and out-neighbours of the vertices. This information suffices to uniquely identify $275$ of the vertices. The remaining $4$ vertices come in two indistinguishable pairs. Therefore
\[
\operatorname{Aut}(G)=\Z_2\oplus\Z_2.
\]
The full computation is available online \cite{github}, \texttt{C. Elegans Automorphisms.nb}.
\end{remark}

\section*{Acknowledgements}

I would like to thank everyone from the Neurotopology Team in Aberdeen for bearing with me while I was working on this project. Special thanks to Ran Levi for support, encouragement, reading parts of the paper and providing useful comments, Jason P. Smith for suggesting an improvement which made my \textsf{maximal-faces} function in \textsc{Mathematica} a little bit faster, Henri Riihim\"aki for a thorough reading of everything and going through it with me section by section and Etienne Lasalle for sharing his note \cite{lasalle} with me. Thanks to the people at the BBP, in particular Kathryn Hess for some early positive feedback, and Martina Scolamiero and Gard Spreemann who did the original directed flag complex and homology computations for C. Elegans in \cite{frontiers}, which have benefited me in many ways. Thanks also to Leon Lampret for many useful discussions about homology and how to compute it over the years, Pedro Concei\c{c}\~ao who also read parts of the draft, and Mark Grant and Irakli Patchkoria for some illuminating discussions regarding Proposition \ref{moorewedge}.

\bibliographystyle{plain}
\bibliography{homotopy}

\appendix

\section{Description of Algorithms and Implementation}
\label{sec:algorithms}

In this section the two procedures used to obtain the main results are described in detail. The ordering of the various lists in the intermediate outputs of these procedures is especially important, as this can affect the order of the collapsing and coning operations and thus affect the outcome for the cases of interest. As such, special attention needs to be paid to it when implementing the procedures; for this reason, as well as to avoid cluttering the presentation, it is described separately in Subsection \ref{sec:ordering}. For example, using the specific ordering of the vertices of C. Elegans described in Section \ref{sec:preprocessing}, the procedure \textsf{cone-and-collapse} arrives at a complex consisting of a single vertex, but altering the ordering of these vertices, simplices in various intermediate complexes or operations used can easily cause the procedure to get stuck in the middle; depending on the case at hand, this could possibly be remedied by adding further subroutines (see Remark \ref{rmk:othermethods}), but this is outside the scope of the paper.

The main two procedures are \textsf{pop-everything} (arising from Lemma \ref{easywedge}) in pseudocode is given as Algorithm \ref{pop-everything} and \textsf{cone-and-collapse} (arising from Lemma \ref{genwedge}) as Algorithm \ref{cone-and-collapse}. These have various subroutines briefly explained below. It should be noted that \textsf{pop-everything} either returns \textbf{true}, in which case we can immediately conclude that the space represented by $S$ is homotopy equivalent to a wedge of spheres, or it returns \textbf{false}, we cannot draw any useful conclusions at all. On the other hand, \textsf{cone-and-collapse} outputs a complex $T$ such that the initial complex $S$ is homotopy equivalent to the wedge of $T$ and a number of spheres. In the favourable case, $T$ will consist of a single vertex, in which case we can immediately conclude that $S$ is a wedge of spheres. In case this does not happen, $T$ can still be considered to be a ``simplification'' of $S$, which can be analysed further by other methods.

The main subroutines of the main two procedures are: \textsf{seq-collapse}, described as Algorithm \ref{seq-collapse}, which performs a sequence of collapses chosen greedily for as long as there are free faces; there is a corresponding Boolean function \textsf{seq-collapsible}, described as Algorithm \ref{seq-collapsible}, which checks if \textsf{seq-collapse} returns a complex consisting of a single vertex; \textsf{select-cells}, described as Algorithm \ref{select-cells}, which is a greedy procedure to find the cells needed to apply Lemma \ref{easywedge}; \textsf{find-good-cycles}, described as Algorithm \ref{find-good-cycles}, which attempts to find cycles that can be coned off; and \textsf{find-good-components}, described as Algorithm \ref{find-good-components}, which attempts to find components of the subcomplex formed by the top-dimensional simplices that can be coned off.

Apart from these, there are some further subroutines: \textsf{nullspace}, described in Algorithm \ref{nullspace}, computes a basis for the nullspace of a matrix, chosen in a specific way based on the Hermite normal form to standardise the output; \textsf{HNF}, described in Algorithm \ref{HNF}, computes the Hermite normal form of a matrix; \textsf{unique-simplices}, described in Algorithm \ref{unique-simplices}, takes a list $C$ of lists $C_i$ of simplices and selects the simplices in each $C_i$ which appear only for that $i$ -- this is useful when coning off the cycles found using \textsf{find-good-cycles} sequentially, because if only popping the uniquely occurring simplices is allowed when testing the conditions of Lemma \ref{genwedge}, this ensures that after each coning operation the conditions allowing the next one remain satisfied; \textsf{pop}, described in Algorithm \ref{pop}, takes a complex and pops a list of simplices in it; \textsf{cone}, described in Algorithm \ref{cone}, takes a complex and cones off a list of subcomplexes in it.

As some of these are very simple, they are treated as black boxes and are thus only specified by what their input and output should be. It should also be noted that the actual implementation uses some additional functions which are not directly relevant to the theoretical description of the procedures.

The description in pseudocode assumes that an ordered simplicial complex (e.g. a directed flag complex) is represented as an ordered list $S$ of (ordered) simplices $s$ that span the complex. Such a list is not assumed to be necessarily closed downwards with respect to taking ordered sublists. In fact, some of these complexes are assumed to be given by a list of their maximal faces (to emphasise this, they are in some cases denoted by $M$). Ordered simplices are assumed to be given as ordered lists of vertices. %%% Cut out a bunch of stuff.

\begin{algorithm}[htb]
\caption{\textsf{pop-everything}}
\label{pop-everything}
\renewcommand{\algorithmicrequire}{\textbf{Input:}}
\renewcommand{\algorithmicensure}{\textbf{Output:}}
\renewcommand{\algorithmicprocedure}{\textbf{Procedure:}}
\begin{algorithmic}[1]
\Require ordered list of simplices $S$
\Ensure \textbf{true} if simplices as in Lemma \ref{easywedge} have been found, \textbf{false} otherwise
\Procedure{}{}
\State set $M=\textsf{seq-collapse}(S)$
\State compute Betti vector $(\beta_0,\ldots,\beta_n)$ of $S$, where $n=\dim S$
\For{$i=1,\ldots,n$}
\If{$\beta_i\neq 0$}
\State set $C_i = \textsf{select-cells}(M,i)$
\If{the length of $C_i$ differs from $\beta_i$}
\State \Return{\textbf{false}}
\EndIf
\EndIf
\EndFor
\State set $C$ to be the list obtained by concatenating all the $C_i$
\State set $A = \textsf{pop}(M,C)$
\State \Return{$\textsf{seq-collapsible}(A)$}
\EndProcedure
\end{algorithmic}
\end{algorithm}

\begin{algorithm}[b]
\caption{\textsf{cone-and-collapse}}
\label{cone-and-collapse}
\renewcommand{\algorithmicrequire}{\textbf{Input:}}
\renewcommand{\algorithmicensure}{\textbf{Output:}}
\renewcommand{\algorithmicprocedure}{\textbf{Procedure:}}
\begin{algorithmic}[1]
\Require ordered list of simplices $S$
\Ensure ordered list $M$ representing the maximal faces of a complex such that the initial complex $S$ is homotopy equivalent to the wedge of $M$ and a number of spheres
\Procedure{}{}
\State set $M=\textsf{seq-collapse}(S)$
\While{\textbf{true}}
\State set $G = \textsf{find-good-cycles}(M)$
\If{$G$ is empty}
\State set $G = \textsf{find-good-components}(M)$
\If{$G$ is empty}
\State \Return $M$
\EndIf
\State set $M=\textsf{seq-collapse}(\textsf{cone}(M,G))$
\EndIf
\State set $M=\textsf{seq-collapse}(\textsf{cone}(M,G))$
\EndWhile
\EndProcedure
\end{algorithmic}
\end{algorithm}

\begin{algorithm}[htb]
\caption{\textsf{seq-collapse}}
\label{seq-collapse}
\renewcommand{\algorithmicrequire}{\textbf{Input:}}
\renewcommand{\algorithmicensure}{\textbf{Output:}}
\renewcommand{\algorithmicprocedure}{\textbf{Procedure:}}
\begin{algorithmic}[1]
\Require ordered list of simplices $S$
\Ensure ordered list $M$ of maximal simplices in the complex obtained after performing a sequence of collapses chosen greedily for as long as there are free faces available
\Procedure{}{}
\State set $M$ to be the list of maximal faces of $S$
\While{$M$ has free faces and more than one vertex}
\State set $t$ to be the \textsf{short-lex}-first free face of $M$ (see Subsection \ref{sec:ordering})
\State collapse $t$ in $M$ along the unique maximal face it is contained in
\State set $M$ to be the list of maximal faces in the complex thus obtained 
\EndWhile
\Return{$M$}
\EndProcedure
\end{algorithmic}
\end{algorithm}

\begin{algorithm}[htb]
\caption{\textsf{seq-collapsible}}
\label{seq-collapsible}
\renewcommand{\algorithmicrequire}{\textbf{Input:}}
\renewcommand{\algorithmicensure}{\textbf{Output:}}
\renewcommand{\algorithmicprocedure}{\textbf{Procedure:}}
\begin{algorithmic}[1]
\Require ordered list of simplices $S$
\Ensure \textbf{true} if $\textsf{seq-collapse}(S)$ consists of a single vertex, \textbf{false} otherwise
%\Procedure{}{}
%\EndProcedure
\end{algorithmic}
\end{algorithm}

\begin{algorithm}[htb]
\caption{\textsf{select-cells}}
\label{select-cells}
\renewcommand{\algorithmicrequire}{\textbf{Input:}}
\renewcommand{\algorithmicensure}{\textbf{Output:}}
\renewcommand{\algorithmicprocedure}{\textbf{Procedure:}}
\begin{algorithmic}[1]
\Require ordered list of simplices $M$ representing the maximal simplices of an ordered simplicial complex, nonnegative integer $d$
\Ensure list $C$ of $d$-simplices in $M$, popping each of which sequentially decreases $\beta_d$ by one and preserves the other Betti numbers
\Procedure{}{}
\State set $C$ to be the empty list
\For{$s$ in $M$}
%\If{$\beta_d(\textsf{pop}(M,\textsf{append}(C,s)))=\beta_d(\textsf{pop}(M,C))-1$ and the remaining Betti numbers are the same}
\State $T = \textsf{pop}(M,C)$
\If{\textsf{pop}ping $s$ in $T$ decreases $\beta_d$ by $1$ and preserves $\beta_i$ for $i\neq d$}
\State append $s$ to $C$
\EndIf
\EndFor
\State \Return{$C$}
\EndProcedure
\end{algorithmic}
\end{algorithm}

\begin{algorithm}[htb]
\caption{\textsf{find-good-cycles}}
\label{find-good-cycles}
\renewcommand{\algorithmicrequire}{\textbf{Input:}}
\renewcommand{\algorithmicensure}{\textbf{Output:}}
\renewcommand{\algorithmicprocedure}{\textbf{Procedure:}}
\begin{algorithmic}[1]
\Require ordered list of simplices $M$ representing the maximal simplices of an ordered simplicial complex
\Ensure ordered list $G$ of sublists of $M$ representing cycles in top dimension to be coned off
\Procedure{}{}
\State set $G$ to be the empty list
\State set $T$ to be the list of all top-dimensional simplices in $M$
\State set $\partial$ to be the simplicial boundary matrix in degree $d=\dim(M)$
\State set $B=\textsf{nullspace}(\partial)$
\State set $C$ to be the list of supports of elements of $B$, i.e. an element $c$ of $C$ is the list of all $d$-simplices that have a nonzero coefficient in the corresponding element $b$ of $B$
\State set $I=\textsf{unique-simplices}(C)$
\For{$c$ in $C$}
\For{$t$ in $I(c)$}
\If{$\textsf{seq-collapsible}(\textsf{pop}(c,t))$}
\State append $c$ to $G$
\State \textbf{break}
\EndIf
\EndFor
\EndFor
\State \Return{$G$}
\EndProcedure
\end{algorithmic}
\end{algorithm}

\begin{algorithm}[htb]
\caption{\textsf{find-good-components}}
\label{find-good-components}
\renewcommand{\algorithmicrequire}{\textbf{Input:}}
\renewcommand{\algorithmicensure}{\textbf{Output:}}
\renewcommand{\algorithmicprocedure}{\textbf{Procedure:}}
\begin{algorithmic}[1]
\Require ordered list of simplices $M$ representing the maximal simplices of an ordered simplicial complex
\Ensure ordered list $G$ of sublists of $M$ representing the components of the subcomplex spanned by the top-dimensional simplices to be coned off
\Procedure{}{}
\State set $G$ to be the empty list
\State set $T$ to be the list of all top-dimensional simplices of $M$
\State set $C$ to be the ordered list whose elements form a partition of $T$ into sublists representing the connected components of $T$
\For{$c$ in $C$}
\If{$\textsf{pop-everything}(c)$}
\State append $c$ to $G$
\EndIf
\EndFor
\State \Return{$G$}
\EndProcedure
\end{algorithmic}
\end{algorithm}

\begin{algorithm}[htb]
\caption{\textsf{nullspace}}
\label{nullspace}
\renewcommand{\algorithmicrequire}{\textbf{Input:}}
\renewcommand{\algorithmicensure}{\textbf{Output:}}
\renewcommand{\algorithmicprocedure}{\textbf{Procedure:}}
\begin{algorithmic}[1]
\Require $(m\times n)$ matrix $A=(a_{ij})$ with integer entries
\Ensure ordered list $B$ representing a basis of the null space of $A$
\Procedure{}{}
\State set $A'$ as the $(m\times n)$ matrix with $(i,j)$-entry given by $a_{i(n+1-j)}$
\State set $A''=\begin{bmatrix} A'\\ I_n \end{bmatrix}$ where $I_n$ is the $(n\times n)$ identity matrix
\State set $H=\textsf{HNF}(A''^T)$ (here $X^T$ denotes the matrix transpose of $X$)
\State write $H=\begin{bmatrix} C & D\\ O & B' \end{bmatrix}$, where $O$ is a $(k\times m)$ zero matrix with maximal $k$
\State set $B=(b_{ij})$ as the $(k\times n)$ matrix with $b_{ij}=b'_{i(n+1-j)}$, where $B'=(b'_{ij})$
\State \Return{$B$ as a list of $k$ rows of length $n$}
\EndProcedure
\end{algorithmic}
\end{algorithm}

\begin{algorithm}[htb]
\caption{\textsf{HNF}}
\label{HNF}
\renewcommand{\algorithmicrequire}{\textbf{Input:}}
\renewcommand{\algorithmicensure}{\textbf{Output:}}
\renewcommand{\algorithmicprocedure}{\textbf{Procedure:}}
\begin{algorithmic}[1]
\Require matrix $A$ with integer entries
\Ensure matrix $H$, the Hermite normal form of $A$ (see Subsection \ref{sec:HNF})
%\Procedure{}{}
%\EndProcedure
\end{algorithmic}
\end{algorithm}

\begin{algorithm}[htb]
\caption{\textsf{unique-simplices}}
\label{unique-simplices}
\renewcommand{\algorithmicrequire}{\textbf{Input:}}
\renewcommand{\algorithmicensure}{\textbf{Output:}}
\renewcommand{\algorithmicprocedure}{\textbf{Procedure:}}
\begin{algorithmic}[1]
\Require ordered list $C$ whose elements $C_i$ ($i=1,\ldots,n$) are ordered lists of simplices
\Ensure ordered list $I$ whose elements $I_i$ ($i=1,\ldots,n$) are ordered lists of simplices and $I_i$ contains precisely the elements of $C_i$ which do not appear in $C_j$ for $j\neq i$
%\Procedure{}{}
%\EndProcedure
\end{algorithmic}
\end{algorithm}

\begin{algorithm}[htb]
\caption{\textsf{pop}}
\label{pop}
\renewcommand{\algorithmicrequire}{\textbf{Input:}}
\renewcommand{\algorithmicensure}{\textbf{Output:}}
\renewcommand{\algorithmicprocedure}{\textbf{Procedure:}}
\begin{algorithmic}[1]
\Require ordered list of simplices $S$ and a sublist $T$
\Ensure ordered list $M$ representing the maximal faces of the simplicial complex obtained by removing the elements of $T$ from $S$ and then adding the codimension $1$ faces of elements of $T$ back to $S$
%\Procedure{}{}
%\EndProcedure
\end{algorithmic}
\end{algorithm}

\begin{algorithm}[htb]
\caption{\textsf{cone}}
\label{cone}
\renewcommand{\algorithmicrequire}{\textbf{Input:}}
\renewcommand{\algorithmicensure}{\textbf{Output:}}
\renewcommand{\algorithmicprocedure}{\textbf{Procedure:}}
\begin{algorithmic}[1]
\Require ordered list of simplices $S$ and an ordered list $G$ of lists $G_i$ ($i=1,\ldots,n$) representing subcomplexes of the ordered simplicial complex represented by $S$
\Ensure ordered list $M$ representing the maximal faces of the ordered simplicial complex obtained by adding the simplices in the lists $G_i'$ to $S$, where each $G_i'$ is obtained from $G_i$ by choosing a new vertex $x_i$ and appending it to each of its elements (in other words $G'_i=G_i\ast\{x_i\}$, a simplicial join); we assume this choice is such that $m<x_1<\ldots<x_n$ where $m$ is the largest vertex of $S$
%\Procedure{}{}
%\EndProcedure
\end{algorithmic}
\end{algorithm}

\subsection{The Hermite normal form}
\label{sec:HNF}

As mentioned earlier, the outcome of various coning operations will depend critically on the choice of subcomplexes to use them on, as well as the ordering of these subcomplexes. A part of the procedure is based on coning off supports of certain cycles. For this reason, it is important to pay special attention to the choice of the cycle basis used. The specific choice of cycle basis used in the procedure is based on the Hermite normal form, which is an analogue of reduced row echelon form for matrices with integer entries. There are various conventions used in the literature as to what properties the Hermite normal form should have, but whichever convention is chosen, the resulting normal form will be unique. The following convention is used (see e.g. \cite[Definition 2.8]{adkins-weintraub} or \cite[Definition 14.8]{bremner}):

\begin{definition}
The matrix $H\in\M_{m\times n}(\Z)$ is said to be in {\em Hermite normal form} if there is a $r\in\{0,\ldots,m\}$ such that the first $r$ rows of $H$ are nonzero and the rest are zero and there are integers $1\leq n_1<n_2<\ldots<n_r\leq m$ such that:
\begin{itemize}
\item $h_{ij}=0$ for $j<n_i$,
\item $h_{in_i}\geq 1$,
\item $0\leq h_{kn_i}<h_{in_i}$ for $k<i$.
\end{itemize}
\end{definition}

In other words, the nonzero rows come before the zero rows, the first nonzero entry (pivot) in each row is positive and strictly to the right of the one in the previous row and the entries above the pivots are nonnegative and less than the pivot. Every matrix with integer entries has a unique Hermite normal form (see Theorems 2.9 and 2.13 in \cite{adkins-weintraub}):

\begin{theorem}
For any matrix $A\in\M_{m\times n}(\Z)$ there is a unique invertible matrix $U\in\GL_{m}(\Z)$ and a matrix $H\in\M_{m\times n}(\Z)$ in Hermite normal form such that $A=UH$. We call $H$ {\em the Hermite normal form of $A$}.
\end{theorem}

This makes the Hermite normal form useful for the purpose of standardising the choices of cycle bases (see Algorithms \ref{find-good-cycles}, \ref{nullspace} and \ref{HNF}).

\subsection{Ordering of the intermediate outputs}
\label{sec:ordering}

The outputs of the algorithms described in pseudocode are ordered in the \textsf{short-lex} order (see \cite[p. 56]{wordproc} or \cite[p. 14]{sipser}), whenever they are ordered lists of simplices (note that this is also the standard ordering of such lists used in \textsc{Mathematica}). In the \textsf{short-lex} order, $s_1<s_2$ whenever either $s_1$ is shorter than $s_2$ or they have the same length and $s_1$ precedes $s_2$ in lexicographic order.

To be specific, the following lists are all assumed to be ordered in \textsf{short-lex}: all directed flag complexes (in particular, the one arising from the C. Elegans graph; in this case a specific ordering of vertices is assumed, see Subsection \ref{sec:preprocessing}), all lists of maximal simplices, this includes the outputs of \textsf{pop}, \textsf{cone}, \textsf{select-cells}, \textsf{seq-collapse} (as well as the list of maximal faces at its beginning and the lists arising from its intermediate collapsing operations), \textsf{cone-and-collapse}, the elements of the lists given by \textsf{find-good-cycles} and \textsf{find-good-components}, as well as the lists of top-dimensional simplices and the lists $c$ that appear as elements of the two lists called $C$ in these two procedures.

The outputs of the functions \textsf{find-good-cycles} and \textsf{find-good-components} are given as lists of lists of simplices, so a brief explanation of how the lists on these lists are ordered among themselves is in order. In the case of \textsf{find-good-cycles}, the ordering arising from the \textsf{nullspace} algorithm is used (while removing some of the rows). In the case of \textsf{find-good-components}, each list in the list has a set of vertices disjoint from the others and the lists are ordered so that the corresponding sets of vertices, ordered from the smallest to largest element, are ordered in \textsf{short-lex}.

The function \textsf{unique-simplices} preserves the ordering in its input lists while removing duplicates. The boundary matrix in Algorithm \ref{find-good-cycles} is assumed to be computed with respect to the lexicographic (\textsf{short-lex}) ordering of the simplices.

\subsection{Some implementation details}
\label{sec:implementation}

The \textsc{Mathematica} functions written to compute the directed flag complex (as well as the corresponding function computing the flag tournaplex) from the directed graph are inspired by the \textsc{Flagser} algorithm \cite{flagser}. The functions used for computing the integral homology of such complexes are based on the algorithms described in \cite{topforcomp}, in particular, to compute torsion, the Smith normal form is used.

Some of the algorithms described in pseudocode are implemented slightly differently than described, however this does not affect their outputs in any way and leads to the same sequence of collapses and coning operations. For instance, in some cases heuristic homology computations have been added to the implementation which allows one to skip certain cycles which have no chance of being spherical. The function used to perform the collapses, \textsf{collapse}, also modifies the values of some other functions, notably \textsf{maximal-faces}, which would be more expensive to compute from scratch. Many things are stored in memory for an additional speed up. The function \textsf{find-good-components} does not use \textsf{pop-everything}, but rather reimplements a part of it, to enable recording the specific simplices which are being popped when verifying that a component may be coned off in the file \texttt{sequence.dat}.

The \textsf{nullspace} function is implemented as described in Algorithm \ref{nullspace}. Initially, however, \textsf{Mathematica}'s built-in \textsf{NullSpace} function was used instead. Based on a degree of experimentation, it appears that these two functions do exactly the same thing for matrices with integer coefficients, which is what led to Algorithm \ref{nullspace} in the first place. There are other differences which the reader is invited to inspect in the code itself \cite{github}.

\subsection{Preprocessing steps}
\label{sec:preprocessing}

The edges in the Excel file \texttt{NeuronConnect.xls} storing the C. Elegans data \cite{celegans} are represented as quadruples, where the first two entries represent the two neurons forming the edge, the third entry represents type of connection and the fourth entry represents the number of synapses (this information is discarded in the construction of the graph). Regarding the ``type of connection'' there are six types: \textsf{S}, \textsf{Sp}, \textsf{R}, \textsf{Rp}, \textsf{EJ} and \textsf{NMJ}. Here, \textsf{EJ} means ``electrical junction'' and \textsf{NMJ} means ``neuromuscular junction''. The remaining four types are the ones used in the construction of the graph, i.e. chemical synapses: \textsf{S} and \textsf{Sp} mean that the first listed neuron sends a chemical synapse to the second listed neuron; \textsf{R} and \textsf{Rp} mean that the first listed neuron receives a chemical synapse from the second listed neuron. Therefore, the graph can be constructed either by taking into account all entries of type \textsf{S} and \textsf{Sp}, or alternatively, all entries of type \textsf{R} and \textsf{Rp}. If doing the latter, one has to be careful if using a case-sensitive extraction mechanism, as two of the neurons are listed as ``\textsf{avfl}'' and ``\textsf{avfr}'' rather than ``\textsf{AVFL}'' and ``\textsf{AVFR}'' (in all the other cases, upper-case names are used), but whatever the choice, it will lead to the same graph $G$. For the actual computation, the \textsf{S} and \textsf{Sp} types were used.

Having extracted the C. Elegans graph from the data, some preprocessing steps are performed before running the procedure \textsf{cone-and-collapse}. First, to avoid getting stuck, it is important how the vertices in the graph $G$ are ordered: select all the entries in the data file with the type \textsf{S} or \textsf{Sp}, then concatenate the ordered pairs of neurons as they appear in these entries (so first neuron of first entry comes first, second neuron of first entry comes second, first neuron of second entry comes third, etc.) From the list obtained in this way, delete all duplicates, only keeping the first occurrence of each neuron. Then reverse all the edges in the resulting graph to obtain a new graph $G^{\mathrm{op}}$ (this operation does not affect the ordering of the vertices or the homotopy type of the directed flag complex). Finally, make sure that the simplices of the directed flag complex $X=\dFl(G^{\mathrm{op}})$ are ordered lexicographically with respect to the ordering of the vertices. In the computation, each vertex was assigned a number from $0$ to $278$ in the described order. The implementation of the directed flag complex then makes sure the simplices are ordered in \textsf{short-lex}.

It is unclear why the vertex ordering described above performs better than some other obvious choices, such as ordering the vertices alphabetically.

\end{document}